\documentclass[a4paper, 12pt]{article}
\pdfoutput=1

\usepackage[T1]{fontenc}
\usepackage[utf8]{inputenc}

\usepackage{amsfonts,amsmath,amssymb,mathtools,dsfont,microtype} 
\usepackage[usenames, dvipsnames]{xcolor} 
\usepackage[noBBpl]{mathpazo} 

\DeclareFontFamily{U} {cmr}{}

\DeclareFontShape{U}{cmr}{m}{n}{
	<-6> cmr5
	<6-7> cmr6
	<7-8> cmr7
	<8-9> cmr8
	<9-10> cmr9
	<10-12> cmr10
	<12-> cmr12}{}

\DeclareSymbolFont{Xcmr} {U} {cmr}{m}{n}
\DeclareMathSymbol{\Omega}{\mathord}{Xcmr}{'012}

\usepackage[labelsep = period, labelfont = bf, justification = centering]{caption}
\usepackage{float,graphicx,subcaption}
\usepackage{booktabs,makecell}
\usepackage{enumitem,mdwlist}
\setlist[itemize]{topsep=0ex,itemsep=0ex,parsep=0.4ex}
\setlist[enumerate]{topsep=0ex,itemsep=0ex,parsep=0.4ex}

\usepackage{tkz-euclide}
\tkzSetUpPoint[fill=black, size = 3pt]
\tkzSetUpLine[color=black, line width=0.6pt]
\tkzSetUpCircle[color=black, line width=0.6pt]

\usepackage[left = 2.5cm, right = 2.5cm, top = 2.5cm, bottom = 2.5cm, headsep = 12pt, headheight = 15pt]{geometry}
\usepackage{parskip}

\definecolor{lightseagreen}{rgb}{0.13, 0.7, 0.67}

\usepackage[hyphens]{url} 
\usepackage[linktoc = all, hidelinks, colorlinks, unicode=true, pagebackref=true]{hyperref} 
\usepackage[capitalise, compress, nameinlink, noabbrev]{cleveref} 
\hypersetup{linkcolor={{blue!70!black}}, citecolor={black}, urlcolor={blue!70!black}}

\usepackage{amsthm,thmtools,thm-restate}
\declaretheorem[name = Theorem, numberwithin = section, style = plain]{theorem}

\declaretheorem[name = Corollary, numberlike = theorem, style = plain]{corollary}

\declaretheorem[name = Definition, numberlike = theorem, style = definition]{definition}
\declaretheorem[name = Lemma, numberlike = theorem, style = plain]{lemma}

\declaretheorem[name = Problem, numberlike = theorem, style = plain]{problem}
\declaretheorem[name = Remark, numberlike = theorem, style = definition]{remark}
\crefname{observation}{Observation}{Observations}
\crefname{conjecture}{Conjecture}{Conjectures}
\crefname{claim}{Claim}{Claims}
\crefname{problem}{Problem}{Problems}

\DeclareFontFamily{U}{matha}{\hyphenchar\font45}
\DeclareFontShape{U}{matha}{m}{n}{
	<5> <6> <7> <8> <9> <10> gen * matha
	<10.95> matha10 <12> <14.4> <17.28> <20.74> <24.88> matha12
}{}
\DeclareSymbolFont{matha}{U}{matha}{m}{n}

\DeclareMathSymbol{\specialuparrow}{\mathrel}{matha}{"D2}
\DeclareMathSymbol{\specialrightarrow}{\mathrel}{matha}{"D1}

\renewcommand*{\backref}[1]{}
\renewcommand*{\backrefalt}[4]{
	\ifcase #1 Not cited.%
	\or $\specialuparrow$#2%
	\else $\specialuparrow$#2%
	\fi%
}

\renewcommand{\epsilon}{\varepsilon}

\renewcommand{\geq}{\geqslant}
\renewcommand{\leq}{\leqslant}
\renewcommand{\emptyset}{\varnothing}
\renewcommand{\subset}{\subseteq}
\renewcommand{\supset}{\supseteq}

\DeclarePairedDelimiter{\abs}{\lvert}{\rvert}

\DeclarePairedDelimiter{\set}{\{}{\}}

\newcommand{\defn}[1]{\textcolor{Maroon}{\emph{#1}}}

\newcommand{\bN}{\mathbb{N}}

\newcommand{\bZ}{\mathbb{Z}}

\newcommand{\cB}{\mathcal{B}}

\newcommand{\cH}{\mathcal{H}}
\newcommand{\cK}{\mathcal{K}}
\newcommand{\cO}{\mathcal{O}}
\newcommand{\cR}{\mathcal{R}}

\title{When $t$-intersecting hypergraphs admit bounded $c$-strong colourings}

\date{\today}

\begin{document}

\author{
Kevin Hendrey\footnotemark[1]\qquad Freddie Illingworth\footnotemark[2]\\ Nina Kam\v{c}ev\footnotemark[3]\qquad Jane Tan\footnotemark[4]
}

\maketitle

\begin{abstract}
    The $c$-strong chromatic number of a hypergraph is the smallest number of colours needed to colour its vertices so that every edge sees at least $c$ colours or is rainbow. 
    We show that every $t$-intersecting hypergraph has bounded $(t + 1)$-strong chromatic number, resolving a problem of Blais, Weinstein and Yoshida.
    In fact, we characterise when a $t$-intersecting hypergraph has large $c$-strong chromatic number for $c\geq t+2$.
    Our characterisation also applies to hypergraphs which exclude sunflowers with specified parameters.
\end{abstract}


\renewcommand{\thefootnote}{\fnsymbol{footnote}} 


\footnotetext[1]{Discrete Mathematics Group, Institute for Basic Science,
Daejeon, South Korea (\textsf{\href{mailto:kevinhendrey@ibs.re.kr}{kevinhendrey@ibs.re.kr}}). Supported by the Institute for Basic Science (IBS-R029-C1)}
\footnotetext[2]{Department of Mathematics, University College London, UK (\textsf{\href{mailto:f.illingworth@ucl.ac.uk}{f.illingworth@ucl.ac.uk}}). Research supported by EPSRC grant EP/V521917/1 and the Heilbronn Institute for Mathematical Research.}
\footnotetext[3]{Department of Mathematics, Faculty of Science, University of Zagreb, Croatia (\textsf{\href{mailto:nina.kamcev@math.hr}{nina.kamcev@math.hr}}). Supported by the Croatian Science Foundation, project number HRZZ-IP-2022-10-5116 (FANAP).}
\footnotetext[4]{All Souls College, University of Oxford, UK (\textsf{\href{mailto:jane.tan@maths.ox.ac.uk}{jane.tan@maths.ox.ac.uk}}).}

\renewcommand{\thefootnote}{\arabic{footnote}} 

\section{Introduction}

The study of hypergraph colourings has long been extensive and fruitful, encompassing not only some fundamental results within combinatorics but also many results that have far-reaching theoretical and practical applications. Part of this richness comes from the fact that there are a multitude of ways in which one can generalise the notion of proper vertex-colourings of graphs to hypergraphs across different settings. Some classical variants of hypergraph colourings are discussed in \cite[Chapter 4]{Berge}, although numerous others have also garnered interest (for instance \cite{BT09, KZ21, PT09}). Among all of these, there are two predominant definitions in the literature: a \defn{weak} colouring of a hypergraph is one for which there are no monochromatic edges so that each edge see at least two different colours, and a \defn{strong} colouring is one for which every edge is \defn{rainbow} meaning no two vertices in an edge have the same colour. 

Interpolating between weak and strong colourings, a \defn{$c$-strong} colouring of a hypergraph $\cH$ is an assignment of colours to the vertices of $\cH$ such that every edge $e\in E(\cH)$ contains vertices with at least $\min\set{c, \abs{e}}$ distinct colours. 
When $c = 2$ this recovers the notion of a weak colouring, whilst a strong colouring corresponds to being $\infty$-strong (or $c$-strong for $c = \max_{e\in E(\cH)} \abs{e}$). The \defn{$c$-strong chromatic number of $\cH$}, denoted \defn{$\chi(\cH, c)$}, is the smallest number of colours needed in a $c$-strong colouring of $\cH$. Despite being a particularly natural third variant, results concerning $c$-strong colourings are surprisingly recent and few. These exist in the context of generalised chromatic polynomials \cite{DBpoly, DBLpoly}, random uniform hypergraphs \cite{Sha21}, and intersecting hypergraphs, the last of which has received the most attention and is also our starting point.


A hypergraph $\cH$ is \defn{$t$-intersecting} if for every pair of edges $e, f \in E(\cH)$ we have $\abs{e \cap f} \geq t$. In 2012, Blais, Weinstein and Yoshida \cite{BWY14} defined the parameter $\chi(t, c)$ to be the minimum number of colours that suffice to $c$-strong colour every $t$-intersecting hypergraph (write $\chi(t, c) = \infty$ if this is unbounded), and posed the problem of determining this value for all combinations of $t$ and $c$. Classical work of Erd\H{o}s and Lov\'asz~\cite{EL75} solves the weak colouring case (that is, $c = 2$) as they showed that 
$\chi(0, 2) = \infty$, $\chi(1, 2) = 3$, and $\chi(t, 2) = 2$ for every $t \geq 2$.

Extending these results, Blais, Weinstein and Yoshida showed that $\chi(t, c)$ is finite when $t \geq c$ or $(t, c) = (1, 2)$ and unbounded when $t \leq c - 2$.
This leaves a conspicuous boundary case, when $t=c-1$.
\begin{problem}[\cite{BWY14}]\label{finiteproblem}
    Determine whether $\chi(c - 1, c)$ is finite or not for every $c > 2$.
\end{problem}

Since being posed, only the $c=3$ case of this problem has been resolved.
 Chung~\cite{Chung2013} showed that $\chi(2, 3)\leq 21$, and  
Colucci and Gy\'{a}rf\'{a}s~\cite{CG13} independently established finiteness with the precise value $\chi(2, 3) = 5$. 

Blais, Weinstein and Yoshida's proof that $\chi(c, c)$ is finite used a random colouring. 
However, Alon~\cite{Alon2013} showed that for every integer $N$ there are $(c - 1)$-intersecting hypergraphs such that the probability that a random $N$-colouring is $c$-strong is arbitrarily small. Thus, uniformly random colourings cannot be used to solve \cref{finiteproblem} and a new approach is needed.
In this paper, we answer the problem fully.
\begin{restatable}{theorem}{mainfiniteness}
\label{thm:mainfiniteness}
    $\chi(c - 1, c)$ is finite for all $c \in \bZ^+$.
\end{restatable}

This theorem completes the characterisation initiated in \cite{BWY14}: $\chi(t, c)$ is finite if and only if $t\geq c - 1$. The bound on $\chi(c - 1, c)$ yielded by our proof is superexponential and we have made no attempt to optimise it.

\Cref{thm:mainfiniteness} will be a consequence of our stronger main theorem, which is more general in two ways and motivated as follows. One reason that $\chi(t, t + \ell) = \infty$ for $\ell \geq 2$ is that one may take a hypergraph with large $\ell$-strong chromatic number (e.g.\ a graph with large chromatic number) and add a fixed set $S$ of $t$ vertices to each hyperedge. We show that this is the \emph{only} reason that a $t$-intersecting hypergraph can have large $(t + \ell)$-strong chromatic number. To this end, for a hypergraph $\cH$, define the \defn{link} of a subset of vertices $S \subset V(H)$ to be the hypergraph \defn{$\cH_S$} with vertex set $V(\cH_S) = V(\cH) \setminus S$ and edge set $\set{e\setminus S \colon S\subset e\in \cH}$. The preceding construction relied on the simple fact that $\chi(\cH, t + \ell) \geq \max_{\abs{S} = t} \chi(\cH_S, \ell)$. Our main theorem, \cref{thm:general}, gives a converse: if $\cH$ is $t$-intersecting and $\chi(\cH_S, \ell)$ is bounded for all sets $S$ of size $t$, then $\chi(\cH, t + \ell)$ is bounded. Now $\chi(\cH', 1) = 1$ for every hypergraph $\cH'$ and so \cref{thm:mainfiniteness} follows immediately.

The second way that \cref{thm:general} generalises \cref{thm:mainfiniteness} is by weakening the assumption that the hypergraph is $t$-intersecting; we will instead require that it does not contain a sunflower of a certain size. Recall that a \defn{sunflower} is a hypergraph with edges $e_1, \dotsc, e_m$ such that if $v\in e_i \cap e_j$ for distinct $i, j \in [m]$, then $v$ is contained in all of the edges. The common intersection $\bigcap_{i \in [m]} e_i$ is the \defn{kernel} of the sunflower, and the \defn{petals} are the pairwise disjoint sets $e_j \setminus \bigcap_{i \in [m]} e_i$ for $j \in [m]$. Note that a hypergraph is $t$-intersecting if and only if it does not contain a sunflower with two petals and kernel of size at most $t - 1$. \Cref{thm:general} replaces `two petals' with `$p$ petals'. That is, we prove:

\begin{restatable}{theorem}{thmgeneral}
\label{thm:general}
    For all non-negative integers $t$, $\ell$, $p$ and $\chi$ with $p \geq 2$ there is some $K$ such that the following holds. Suppose that $\cH$ is a hypergraph such that
    \begin{itemize}[noitemsep]
        \item $\cH$ does not contain a sunflower with $p$ petals and kernel of size at most $t - 1$\textup{;}
        \item for all sets $S$ of $t$ vertices, $\chi(\cH_S, \ell) \leq \chi$.
    \end{itemize}
    Then $\chi(\cH, t + \ell) \leq K$.
\end{restatable}

We will prove \cref{thm:general} in \cref{sec:mainproof} after first introducing the key objects and ideas in \cref{sec:setup}.
As remarked above, the second hypothesis is necessary. The first hypothesis is also necessary in the sense that $t - 1$ cannot be replaced by $t - 2$. Indeed, let $G$ be a graph of large chromatic number, $T$ a set of $t - 1$ new vertices, and $\cH$ the $(t + 1)$-uniform hypergraph whose edges are $T \cup e$ for $e \in E(G)$. Then $\chi(\cH, t + 1) \geq \chi(G)$, the link of each set of $t$ vertices is an independent set, and every sunflower has kernel of size at least $t - 1$. 

\Cref{thm:general} gives a qualitative characterisation of when a $t$-intersecting hypergraph has large $(t + \ell)$-strong chromatic number: when some set of $t$ vertices has a link whose $\ell$-strong chromatic number is large. Defining $\chi(\cH, t, \ell)$ to be the largest $\ell$-strong chromatic number of a link of $t$ vertices, we have $\chi(\cH, t, \ell) \leq \chi(\cH, t + \ell) \leq f_{t, \ell}(\chi(\cH, t, \ell))$ for some function $f_{t, \ell}$. In \cref{sec:extremal} we strengthen this quantitatively by giving both lower and upper bounds for $f_{t, \ell}$. In particular, we show that $f_{t, \ell}(x) = x^{\Theta(t^{\ell - 2})}$ for fixed $\ell$ and growing $t$.

\section{Proof outline and definitions}
\label{sec:setup}

In this section, we introduce the key objects used in the proof of \cref{thm:general} and end with a proof sketch.

Throughout this paper, we use the term \defn{colouring} to mean vertex-colouring. Explicitly, a colouring of a hypergraph $\cH$ is a function $c \colon V(\cH) \to \bN$. On the few occasions we wish to colour edges of a graph or hypergraph, we will use the term \defn{edge-colouring}.

Suppose that $\cH$ is a hypergraph as in the statement of the theorem and that is has very large $(t + \ell)$-strong chromatic number. We will consider various colourings of the vertices of $\cH$ and will need to be able to combine them. We do this with product colourings.

Given colourings $c_1, \dotsc, c_m \colon V(\cH) \to \bN$, their \defn{product colouring} $c^\times \colon V(\cH) \to \bN^m$ is the colouring given by
\begin{equation*}
    c^\times(v) = (c_1(v), \dotsc, c_m(v)).
\end{equation*}
Observe that $\abs{c^\times(V(\cH))}\leq \prod_{i=1}^{m}\abs{c_i(V(\cH))}$. 

The next idea is that any sequence of edges $e_1, \dotsc, e_L \in \cH$ split $V(\cH)$ into regions as follows.

\begin{definition}
\label{def:regions}
    For edges $e_1, e_2, \dotsc, e_L$ of a hypergraph $\cH$, the set of \defn{regions formed by $e_1$, \ldots, $e_L$}, denoted $\cR(e_1, \dotsc, e_L)$, is the set of non-empty intersections $\bigcap_{i \in I} e_i \cap \bigcap_{i \in [L] \setminus I} (V(\cH) \setminus e_i)$ where $I \subset \set{1, \dotsc, L}$.
\end{definition}

Note that regions form a partition of the vertex-set $V(\cH)$ into non-empty sets. 

Consider the following `regional' colouring of $\cH$: we colour the vertices so that each region $R \in \cR(e_1, \dotsc, e_L)$ receives $t + \ell$ colours (if the region has size less than $t + \ell$, then it is given a rainbow colouring) and we use different colour palettes on each region. The number of colours used in this colouring is bounded (at most $(t + \ell) 2^L$). This colouring may give some edges $t + \ell$ colours (these edges have now been coloured appropriately and can be subsequently ignored). However, since $\chi(\cH, t + \ell)$ is very large, the sub-hypergraph $\cH'$ of $\cH$ consisting of those edges receiving fewer than $t + \ell$ colours must still have large $(t + \ell)$-strong chromatic number. Which edges are in $\cH'$? Every such edge must intersect fewer than $t + \ell$ different regions. Furthermore, every such edge cannot entirely contain a region (we use induction and the properties of $\cH$ to deal with edges that entirely contain regions of size less than $t + \ell$). Motivated by this we make the following definition.

\begin{definition}
    A sequence of edges $e_1, e_2, \dotsc, e_\ell$ is \defn{$k$-split-degenerate} if for each $1 \leq j \leq \ell$:
    \begin{itemize}
        \item The edge $e_j$ intersects at most $k$ regions in $\cR(e_1, \dotsc, e_{j - 1})$;
        \item The edge $e_j$ does not contain any region $R \in \cR(e_1, \dotsc, e_{j - 1})$
    \end{itemize}
\end{definition}

\begin{remark}\label{rmk:ksplit}
    Every subsequence of a $k$-split-degenerate sequence is $k$-split-degenerate.
\end{remark}

By iterating the argument of the previous paragraph we show (\cref{lemma:findksplitdegen}) that, provided $\chi(\cH, t + \ell)$ is sufficiently large, $\cH$ contains a long $(t + \ell - 1)$-split-degenerate sequence.

While split-degenerate sequences are quite structured we clean them up further, using Ramsey's theorem, to show (\cref{lemma:findflower}) that they contain the following key structure.

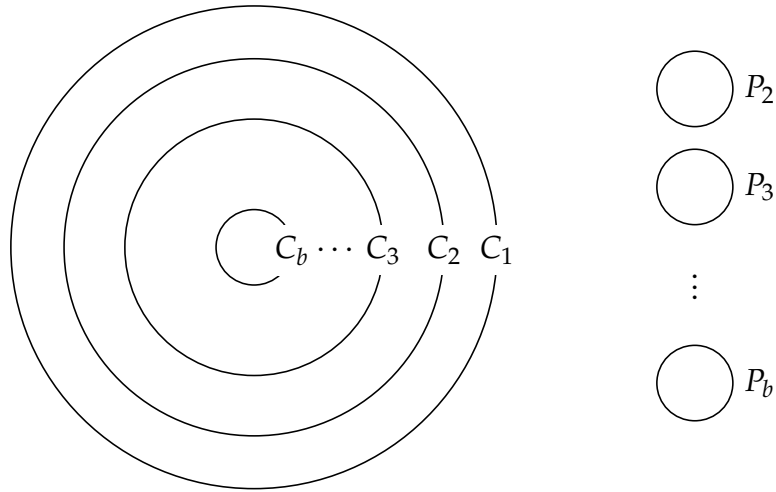
\begin{figure}[H]
    \centering
    \begin{tikzpicture}
        \tkzDefPoint(0,0){O}
        \tkzDefPoint(0.5,0){Cb}
        \tkzDefPoint(1.1,0){dots}
        \tkzDefPoint(1.7,0){C3}
        \tkzDefPoint(2.5,0){C2}
        \tkzDefPoint(3.2,0){C1}

        \tkzDefPoint(5.8,2.1){Q2}
        \tkzDefShiftPoint[Q2](0.5,0){P2}
        \tkzDefPoint(5.8,0.8){Q3}
        \tkzDefShiftPoint[Q3](0.5,0){P3}
        \tkzDefPoint(5.8,-0.5){dots1}
        \tkzDefPoint(5.8,-1.8){Qb}
        \tkzDefShiftPoint[Qb](0.5,0){Pb}

        \tkzDrawCircles(O,Cb O,C3 O,C2 O,C1 Q2,P2 Q3,P3 Qb,Pb)
        
        \tkzLabelPoint[fill=white, shift={(0,0.3)}](C1){$C_1$}
        \tkzLabelPoint[fill=white, shift={(0,0.3)}](C2){$C_2$}
        \tkzLabelPoint[fill=white, shift={(0,0.3)}](C3){$C_3$}
        \tkzLabelPoint[fill=white, shift={(0,0.3)}](Cb){$C_b$}
        \tkzLabelPoint[right](P2){$P_2$}
        \tkzLabelPoint[right](P3){$P_3$}
        \tkzLabelPoint[shift={(0,0.5)}](dots1){$\vdots$}
        \tkzLabelPoint[right](Pb){$P_b$}
        \tkzLabelPoint[shift={(0,0.15)}](dots){$\dots$}
    \end{tikzpicture}
    \caption{A $b$-bromeliad $(C_1, C_2 \sqcup P_2, \dots, C_b \sqcup P_b)$ witnessed by $(\set{C_1, \emptyset}, \set{C_2, P_2}, \dots, \set{C_b, P_b})$.}
\end{figure}

\begin{definition}
    \label{def:BF}
    Given a hypergraph $\cH$, a sequence of edges $\cB = (e_1, \dotsc, e_b)$ of $\cH$ is a \defn{$b$-bromeliad} if there are sets $P_i$ (\defn{petals}) and $C_i$ (\defn{cores}) such that
    \begin{itemize}
        \item for each $i$, $C_i$ and $P_i$ partition $e_i$;
        \item $e_1 = C_1 \supsetneq C_2 \supsetneq \dotsb \supsetneq C_b \neq \emptyset$;
        \item $P_1, \dotsc, P_b, C_1$ are pairwise disjoint.
    \end{itemize}
    We say that the sequence of partitions $(\set{C_1, P_1},\dots ,\set{C_b, P_b})$ witnesses that $\cB$ is a $b$-bromeliad.
\end{definition}

\begin{remark}\label{rmk:Bajan}
    Note that, for all $1\leq i < j \leq b$, $C_j = e_i \cap e_j$ and $P_j = e_j \setminus e_i$. Furthermore, $C_1 = e_1$ and $P_1 = \emptyset$.
    In particular, there is a unique witness for each bromeliad.
\end{remark}

We call $e_1$ the \defn{outer edge} of the bromeliad, and  $C_2 = e_2 \cap e_1$ the \defn{crown} of the bromeliad. A crucial ingredient at the end of our argument is the poset consisting of the family of $b$-bromeliads in a fixed hypergraph with partial order given by
\begin{equation*}
    (e_1, \dotsc, e_b) \prec (e_1' \dotsc, e_b') \text{\quad if \quad } \abs{e_2 \cap e_1} < \abs{e_2'\cap e_1'}.
\end{equation*}

\Cref{lemma:findksplitdegen,lemma:findflower} combine to show (\cref{cor:BF-compatible-gen}) that a hypergraph $\cH$ as in the statement of \cref{thm:general} contains a large Bromeliad. Iterating \cref{cor:BF-compatible-gen} allows us to construct a nested sequence of subhypergraphs $\cH_1 \supseteq \cH_2 \supseteq \dotsb \supseteq \cH_m$ and a sequence of $(t + \ell)$-bromeliads $\cB_1, \dotsc, \cB_m$ where, for each $i$, $\cB_i$ is minimum with respect to the partial order for $\cH_i$. These bromeliads will have additional properties (guaranteed by \cref{lemma:pruning}) which allow us to find a `diagonal' bromeliad (consisting of at most one edge from each of the $\cB_i$) which, for some $i$, precedes $\cB_i$ in the partial order given by $\cH_i$. This contradicts the minimality of $\cB_i$ and so shows that $\chi(\cH, t + \ell)$ is bounded.

\section{The proof of \texorpdfstring{\cref{thm:general}}{main theorem}}
\label{sec:mainproof}

We first show that the assumptions in \cref{thm:general} guarantee the existence of a long split-degenerate sequence. 

\begin{lemma}[Finding long split-degenerate sequences]\label{lemma:findksplitdegen}
    Let $t$, $\ell$, $p$, $\chi$, $L$ be positive integers, and assume that~\cref{thm:general} holds for $t_{\ref{thm:general}} = t-1$ and the given values of $\ell, p, \chi$. Then there is a positive integer $K_{\ref{lemma:findksplitdegen}}$ such that the following holds. Suppose that $\cH$ is a hypergraph and $e_1, \dotsc, e_L$ are edges of $\cH$ such that
    \begin{itemize}
        \item $\cH$ does not contain a sunflower with $p$ petals and kernel of size at most $t - 1$\textup{;}
        \item for all sets $S$ of $t$ vertices, $\chi(\cH_S, \ell) \leq \chi$\textup{;}
        \item the sequence $(e_1, \dotsc, e_L)$ is $(t + \ell - 1)$-split-degenerate\textup{;}
        \item $\chi(\cH, t + \ell) \geq K_{\ref{lemma:findksplitdegen}}$.
    \end{itemize}
    Then there is an edge $e_{L + 1} \in \cH$ such that the sequence $(e_1, \dotsc, e_{L + 1})$ is $(t + \ell - 1)$-split-degenerate.
\end{lemma}

\begin{proof}
    Let $K \coloneqq K(t-1, \ell, p, \chi)$ be the constant obtained by applying~\cref{thm:general} with the specified parameters, and take $K_{\ref{lemma:findksplitdegen}} = (t + \ell) \cdot 2^LK^{(t+\ell - 1)2^L}+ 1$.

    We begin by defining some colourings of $\cH$. The edges that have fewer than $t + \ell$ colours with respect to the product of these colourings will extend the $(t + \ell - 1)$-split-degenerate sequence.
    
    Let $c_\cR$ be a \defn{$(t + \ell)$-regional colouring of $\cH$} with respect to $e_1, \dotsc, e_L$, defined as follows:
    \begin{itemize}
        \item for each region $R \in \cR(e_1, \dotsc, e_L)$ of size less than $t + \ell$, assign each vertex in $R$ a new colour;
        \item for each region $R \in \cR(e_1, \dotsc, e_L)$ of size at least $t + \ell$, take an arbitrary partition of $R$ into $t + \ell$ non-empty parts, and for each part assign one new colour to all the vertices of that part.
    \end{itemize}
    Let us call a region $R \in \cR(e_1, \dotsc, e_L)$ \defn{small} if $\abs{R} \leq t + \ell - 1$, and \defn{large} otherwise. Note that each small region is rainbow-coloured by $c_\cR$ and $t + \ell$ colours appear on each large region.
    The colouring $c_\cR$ uses at most $(t + \ell) \cdot 2^L$ colours. 

    We now also define a colouring $c_v$ of $\cH$ for each vertex $v$ in  a small region $S \in \cR(e_1, \dotsc, e_L)$. We will apply the induction hypothesis to the link hypergraph $\cH_v$.\footnote{$\cH_v$ denotes the link $\cH_S$ where $S = \set{v}$.} To this end, note that $\cH_v$ does not contain a sunflower with $p$ petals and kernel of size $t - 2$. Moreover, for every $S' \subset V(\cH)\setminus \set{v}$ with $\abs{S'} = t - 1$, the link graph of $S'$ in $\cH_v$ is $\cH_{\set{v}\cup S'}$, which satisfies
    $\chi(\cH_{\set{v} \cup S'}, \ell) \leq \chi$ by assumption. Thus $\cH_v$ satisfies the hypothesis of~\cref{thm:general} with $t_{\ref{thm:general}} = t - 1$, and so $\chi(\cH_v, (t-1) + \ell) \leq K$. Let $c_v$ be a $K$-colouring of $\cH$ whose restriction to $\cH_v$ is a $(t + \ell - 1)$-strong colouring.
    
    Let $c^\times$ be the product colouring of all of the colourings $c_v$ (over vertices $v$ in small regions) together with the regional colouring $c_\cR$. There are at most $(t+\ell - 1)2^L$ vertices $v$ in small regions and each colouring $c_v$ uses at most $K$ colours, so $c^\times$ uses at most $(t+\ell)2^{L}K^{(t+\ell - 1)2^L}$ colours. Since $\chi(\cH, t + \ell) > (t+\ell)2^{L}K^{(t+\ell - 1)2^L}$, $c^\times$ is not a $(t + \ell)$-strong colouring and so there is an edge $e'$ which sees at most $t + \ell - 1$ colours.

    First note that $e'$ intersects at most $t + \ell - 1$ regions of $\cR(e_1, \dotsc, e_L)$ since it picks up at least one new colour in the regional colouring from each region that it intersects. In addition, since all large regions have $t + \ell$ different colours in the regional colouring, $e'$ cannot contain any large region. 
    Finally, suppose that $e'$ contains some vertex $v$ of a small region. Then $c_v$ ensures that $c^\times$ assigns at least $t+\ell - 1$ colours to $e' \setminus \set{v}$, and $c_{\cR}$ ensures that each of those colours are distinct from $c^\times(v)$. Hence, $e'$ receives at least $t + \ell$ colours from $c^\times$, a contradiction. 

    Thus $e'$ intersects at most $t + \ell - 1$ regions from $\cR(e_1, \dotsc, e_L)$ and does not contain any region of $\cR(e_1, \dotsc, e_L)$. Hence $e_{L + 1} \coloneqq e'$ extends the $(t + \ell - 1)$-split-degenerate sequence $(e_1, \dotsc, e_L)$.
\end{proof}

We now show that sufficiently long split-degenerate sequences contain long bromeliads as subsequences.
\begin{lemma}[split-degenerate sequence to bromeliad]\label{lemma:findflower}
    For all positive integers $b$, $k$ and $p$ there is a positive integer $f_{\ref{lemma:findflower}}(b, k, p)$ such that the following holds. If $e_1$, \ldots, $e_{f_{\ref{lemma:findflower}}(b, k, p)}$ is a $k$-split-degenerate sequence that does not contain a matching of size $p$, then there are integers $1 \leq a(1) < a(2) < \dotsb < a(b) \leq f_{\ref{lemma:findflower}}(b, k, p)$ such that $e_{a(1)}, e_{a(2)}, \dotsc, e_{a(b)}$ is a $b$-bromeliad.
\end{lemma}

\begin{proof}
    Let $G$ be the graph with vertex-set $\set{e_1, \dotsc, e_{f_{\ref{lemma:findflower}}(b, k, p)}}$ and with an edge between $e_i$ and $e_j$ if $e_i \cap e_j \neq \emptyset$. Since $e_1, \dotsc, e_{f_{\ref{lemma:findflower}}(b, k, p)}$ does not contain a matching with $p$ edges, $G$ does not contain an independent set of size $p$. By the Erd\H{o}s-Szekeres bound for Ramsey's theorem, $G$ contains a clique of size at least $f(b, k) \coloneqq f_{\ref{lemma:findflower}}(b, k, p)^{1/p}$. The subsequence of $e_i$ in this clique is $k$-split-degenerate (by \cref{rmk:ksplit}) and is intersecting (i.e.\ $e_i \cap e_j \neq \emptyset$ for all $i$, $j$). By relabelling, we may assume that $e_1, \dotsc, e_{f(b, k)}$ is an intersecting $k$-split-degenerate sequence.

    Consider the complete 3-uniform hypergraph $\cK^{(3)}$ on vertex set $\set{e_1, \dotsc,\allowbreak e_{f(b, k)}}$, and define an edge-colouring of this hypergraph using 2 colours as follows. For $i_1 < i_2 < i_3$, the edge $e_{i_1} e_{i_2} e_{i_3}$ is
    \begin{itemize}
        \item blue if $(e_{i_1} \triangle e_{i_2}) \cap e_{i_3} \neq \emptyset$;
        \item red if $(e_{i_1} \triangle e_{i_2}) \cap e_{i_3} = \emptyset$;
    \end{itemize}
    where $\triangle$ denotes the symmetric difference. 
    
    Suppose that there is a large blue clique on vertices $e_{j_1}$, $e_{j_2}$, \ldots, $e_{j_\ell}$ where $1 \leq j_1 < j_2 < \dotsb < j_\ell \leq f(b, k)$. We will show that such a clique cannot occur. To derive a contradiction, take an edge-colouring of the complete graph $\cK^{(2)}$ on vertex set $\set{j_1, \dotsc, j_{\ell - 1}}$ as follows. First, fix $1 \leq \alpha < \beta \leq \ell - 1$. Then $e_{j_\alpha} \triangle e_{j_\beta}$ is the union of some regions in $\cR(e_{j_1}, e_{j_2}, \dotsc, e_{j_{\ell - 1}})$. Since $e_{j_\alpha} e_{j_\beta} e_{j_\ell}$ is a blue edge, $e_{j_\ell}$ intersects $e_{j_\alpha} \triangle e_{j_\beta}$ and so there is some region $R \in \cR(e_{j_1}, e_{j_2}, \dotsc, e_{j_{\ell - 1}})$ such that $R \subset e_{j_\alpha} \triangle e_{j_\beta}$ and $e_{j_\ell}$ intersects $R$. Colour edge $j_\alpha j_\beta$ with colour $R$ (if there is more than one such region pick one arbitrarily). 
    
    By \cref{rmk:ksplit}, the sequence $e_{j_1}$, $e_{j_2}$, \ldots, $e_{j_\ell}$ is $k$-split-degenerate, and hence there are at most $k$ regions in $\cR(e_{j_1}, e_{j_2}, \dotsc, e_{j_{\ell - 1}})$ that $e_{j_\ell}$ intersects. Thus there are at most $k$ choices for $R$, so the colouring of the $\cK^{(2)}$ uses at most $k$ colours. By Ramsey's theorem, provided that $\ell$ is sufficiently large, there is a monochromatic triangle $j_\alpha j_\beta j_\gamma$. That is, there is a region $R$ such that $R \subset e_{j_\alpha} \triangle e_{j_\beta}$, $R \subset e_{j_\beta} \triangle e_{j_\gamma}$, $R \subset e_{j_\gamma} \triangle e_{j_\alpha}$, and $e_{j_\ell}$ intersects $R$. However, $(e_{j_\alpha} \triangle e_{j_\beta}) \cap (e_{j_\beta} \triangle e_{j_\gamma}) \cap (e_{j_\gamma} \triangle e_{j_\alpha}) = \emptyset$ (this is true for any three sets), so $R$ is empty, which is a contradiction. Thus there is no large blue clique in the colouring of $\cK^{(3)}$.

    Hence, by Ramsey's theorem, if $f(b, k)$ is sufficiently large then there must be a red clique $e_{a(1)}$, $e_{a(2)}$, \ldots, $e_{a(b)}$ in $\cK^{(3)}$ where $1 \leq a(1) < a(2) < \dotsb < a(b) \leq f(b, k)$. For $i \in \set{1, \dotsc, b}$, let
    \begin{align*}
        C_i & = e_{a(i)} \cap e_{a(1)}, \\
        P_i & = e_{a(i)} \setminus e_{a(1)}.
    \end{align*}
    We now show that $(\set{C_1, P_1}, \dots, \set{C_b, P_b})$ witnesses that $e_{a(1)}$, $e_{a(2)}$, \ldots, $e_{a(b)}$ is a $b$-bromeliad. 
    Firstly, since $e_1, \dotsc, e_{f(b, k)}$ is intersecting, $C_b \neq \emptyset$. Secondly, $C_1 = e_{a(1)}$ by definition. Thirdly,
    \begin{align*}
        C_{i + 1} \setminus C_i & = (e_{a(i + 1)} \cap e_{a(1)}) \setminus (e_{a(i)} \cap e_{a(1)}) = (e_{a(i + 1)} \cap e_{a(1)}) \setminus e_{a(i)} \\
        & = (e_{a(1)} \setminus e_{a(i)}) \cap e_{a(i + 1)} \subset (e_{a(1)} \triangle e_{a(i)}) \cap e_{a(i + 1)} = \emptyset,
    \end{align*}
    where the final equality is automatic for $i = 1$ and uses that the edge $e_{a(1)} e_{a(i)} e_{a(i + 1)}$ is red for $i \geq 2$. Thus $C_{i + 1} \subset C_i$ for each $i$. Next note that $C_i = e_{a(i)} \cap e_{a(1)}$ is the union of some regions in $\cR(e_{a(1)}, \dotsc, e_{a(i)})$. Choose one such region $R \subset C_i$. Since $e_{a(1)}, e_{a(2)}, \dotsc, e_{a(i + 1)}$ is $k$-split-degenerate by \cref{rmk:ksplit}, the edge $e_{a(i + 1)}$ does not contain $R$. It follows that $R \not\subset C_{i + 1}$, and so $C_i \neq C_{i + 1}$.
    Thus $C_{i + 1} \subsetneq C_i$ for each $i$.
    
    It just remains to verify pairwise disjointness of the petals and $C_1$. Since $C_1 = e_{a(1)}$, it is clear that this is disjoint from each $P_i$ by their definition. For $1 \leq i < j \leq b$, we have
    \begin{align*}
        P_i \cap P_j & = (e_{a(i)} \setminus e_{a(1)}) \cap (e_{a(j)} \setminus e_{a(1)}) \\
        & = (e_{a(i)} \setminus e_{a(1)}) \cap e_{a(j)} \subset (e_{a(1)} \triangle e_{a(i)}) \cap e_{a(j)} = \emptyset,
    \end{align*}
    where the last equality again uses the fact that the edge $e_{a(1)} e_{a(i)} e_{a(j)}$ is red. This completes the verification that $e_{a(1)}$, $e_{a(2)}$, \ldots, $e_{a(b)}$ is a $b$-bromeliad, as required.
\end{proof}

Let $(e_1, \dotsc, e_L)$ be a $k$-split-degenerate sequence on $V(\cH)$. We say that a set of edges $\cH'$ is \defn{$k$-compatible} (or simply \defn{compatible} when $k$ is clear from context)  with a sequence of edges $e_1, \dotsc, e_L$ if $e_1, \dotsc, e_L, e'$ is a $k$-split-degenerate sequence for every $e' \in \cH'$. This definition is usually applied in a situation where $\cH'$ is the edge set of a bromeliad.

\begin{corollary}
    \label{cor:BF-compatible-gen}
    Let $t$, $\ell$, $p$, $\chi$, $L$, $b$ be positive integers, and assume that \cref{thm:general} holds for all $t_{\ref{thm:general}} < t$ and the given values of $\ell$, $p$, $\chi$. Then there is some positive integer $K_{\ref{cor:BF-compatible-gen}}$ such that the following holds. Suppose that $\cH$ is a hypergraph and $e_1, \dotsc, e_L$ are edges of $\cH$ such that
    \begin{itemize}[noitemsep]
        \item $\cH$ does not contain a sunflower with $p$ petals and kernel of size at most $t - 1$\textup{;}
        \item for all sets $S$ of $t$ vertices, $\chi(\cH_S, \ell) \leq \chi$\textup{;}
        \item the sequence $(e_1, \dotsc, e_L)$ is $(t + \ell - 1)$-split-degenerate\textup{;}
        \item $\chi(\cH, t + \ell) \geq K_{\ref{cor:BF-compatible-gen}}$.
    \end{itemize}
    Then $\cH$ contains a $b$-bromeliad which is $(t + \ell - 1)$-compatible with $(e_1, \dotsc, e_L)$.
\end{corollary}

\begin{proof}
    Set $k = t + \ell - 1$ and let $M \coloneqq f_{\ref{lemma:findflower}}(b, k, p)$ be the integer guaranteed by \cref{lemma:findflower}.

    Provided $K_{\ref{cor:BF-compatible-gen}}$ is sufficiently large, we may repeatedly apply \cref{lemma:findksplitdegen} to find edges $e_{L + 1}$, \ldots, $e_{L + M}$ of $\cH$ such that $(e_1, \dotsc, e_{L + M})$ is $k$-split-degenerate. The hypergraph $\cH$ does not contain a sunflower with $p$ petals and kernel of size at most $t - 1$, so it does not contain a matching of size $p$.
    Thus, by the definition of $M$, the $k$-split-degenerate sequence $(e_{L + 1}, \dotsc, e_{L + M})$ contains a $b$-bromeliad as a subsequence. This $b$-bromeliad is $k$-compatible with $(e_1, \dotsc, e_L)$.
\end{proof}

Now that we can find a $b$-bromeliad $\cB$ in $\cH$ by \cref{cor:BF-compatible-gen}, our next lemma shows that there is a sub-hypergraph $\cH'$ of $\cH$ that has large $(t + \ell)$-strong chromatic number and whose edges have the additional property that they are disjoint from one of the petals of $\cB$.

\begin{lemma}[Pruning to good intersection pattern]\label{lemma:pruning}
    For all $b$ and $r$ there is some integer $K_{\ref{lemma:pruning}}$ such that the following holds. Let $\cH$ be a hypergraph with $\chi(\cH, b - 1) \geq K_{\ref{lemma:pruning}}$, and let $\cB = (e_1, \dotsc, e_b)$ be a $b$-bromeliad in $\cH$. Then there is a sub-hypergraph $\cH'$ of $\cH$ and an edge $e_j$ of $\cB$ with $j \geq 2$ such that $\chi(\cH', b - 1) \geq r$ and every edge of $\cH'$ is disjoint from the petal of $e_j$.
\end{lemma}

\begin{proof}
    Set $K_{\ref{lemma:pruning}} = b (r - 1)^{b - 1} + 1$. For each $j \in \set{2, \dotsc, b}$, let $\cH_j$ denote the sub-hypergraph of $\cH$ consisting of edges in $\cH$ that are disjoint from the petal $P_j$ of $e_j$. If there is a $j$ such that $\chi(\cH_j, b - 1) \geq r$, then we may take $\cH' = \cH_j$. So suppose not, and for each $j$ choose a $(b - 1)$-strong colouring of $\cH_j$ using at most $r - 1$ colours, and extend it arbitrarily to an $(r - 1)$-colouring $c^j$ of $\cH$. Let $c_P$ be the $b$-colouring of $\cH$ such that for each $j\in \set{2, \dotsc, b}$ every vertex of $P_j$ gets colour $j$ and all vertices not in petals of $\cB$ get colour $1$.
    
    Let $c^\times$ denote the product of $c_P$ and all of the $c^j$. This uses less than $K_{\ref{lemma:pruning}}$ colours and so there is some edge $e$ that sees less than $b - 1$ colours. Now if $e$ intersects all of $P_2, \dotsc, P_b$, then it receives at least $b - 1$ colours from $c_P$, which is a contradiction. Thus $e$ must be disjoint from some $P_j$ where $j \in \set{2, \dotsc, b}$. But then $e \in \cH_j$, and so $e$ received at least $b - 1$ colours from $c^j$, a contradiction. Hence, there is some $j \in \set{2, \dotsc, b}$ with $\chi(\cH_j, b - 1) \geq r$, as required.
\end{proof}

We are now ready to prove our main theorem. 

\thmgeneral*
\begin{proof}
    We proceed by induction on $t$. When $t = 0$, the theorem holds with $K = \chi$ since the second bullet point just becomes $\chi(\cH, \ell) \leq \chi$.

    So let $t \geq 1$, and assume that the theorem holds for all $t' < t$. Let $k \coloneqq t + \ell - 1$, let $b \coloneqq t + \ell + 1$, and let $\Phi \coloneqq f_{\ref{lemma:findflower}}(b + 1, k, p)$ be the integer guaranteed by \cref{lemma:findflower}. Choose integers $1  \ll \chi_\Phi \ll \dots \ll \chi_0$ (we will see how large these should be later).
    Suppose for a contradiction that $\chi(\cH, t + \ell) > \chi_0$. Note that the first bullet point of the theorem statement implies that no subhypergraph of $\cH$ contains a matching of size $p$.

    We will inductively construct a sequence of hypergraphs $\cH = \cH_0 \supset \cH_1 \supset \dotsb \supset \cH_\Phi$, a sequence of $b$-bromeliads $\cB_1, \dotsc, \cB_\Phi$, a $k$-split-degenerate sequence of edges $(e_1, \dotsc, e_\Phi)$ of $\cH$ such that the following hold for all $j \in \set{1, \dotsc, \Phi}$:
    \begin{enumerate}[label={(\arabic*)}]
        \item\label{item:prunededge} $e_j \in \cB_j \subset \cH_{j - 1}$ and $e_j$ is not the outer edge of $\cB_j$;
        \item\label{item:compatible} $\cB_j$ is a $b$-bromeliad in $\cH_{j - 1}$ that is $k$-compatible with $(e_1, \dotsc, e_{j - 1})$;
        \item\label{item:minimal} subject to \ref{item:compatible}, $\cB_j$ is $\prec$-minimal (see \cref{def:BF});
        \item\label{item:disjointpetals} every edge of $\cH_j$ is disjoint from the petal of $e_j$ in $\cB_j$;
        \item\label{item:Uparrow} $\chi(\cH_j, b - 1) > \chi_j$.
    \end{enumerate}
    For $j\in \set{1,\dots,\Phi}$, our construction proceeds as follows. Provided $\chi_{j - 1}$ is sufficiently large, applying \cref{cor:BF-compatible-gen} to $\cH_{j - 1}$ and $(e_1, \dotsc, e_{j - 1})$ (this sequence of edges will be empty for $j = 1$) gives the existence of a $b$-bromeliad in $\cH_{j - 1}$ that is $k$-compatible with $(e_1, \dotsc, e_{j - 1})$. Let $\cB_j$ be a $\prec$-minimal one. Then, provided $\chi_{j - 1}$ is also sufficiently large compared to $\chi_j$, applying \cref{lemma:pruning} to $\cH_{j - 1}$ and $\cB_j$ gives an edge $e_j$ of $\cB_j$ (that is not the outer edge) and a subhypergraph $\cH_j$ of $\cH_{j - 1}$ that satisfy properties \ref{item:disjointpetals} and \ref{item:Uparrow}.\footnote{At this point, we have covered all of the steps that will ultimately affect the bound on $\chi(c-1, c)$ given by this proof. Note that the number of times we need to apply \cref{lemma:pruning} comes from \cref{lemma:findksplitdegen} which uses multiple applications of Ramsey's Theorem. Using known bounds for the Ramsey functions involved, it follows that our bound is of order $2^{2^{2^{2^{\textrm{poly}(c)}}}}$.}

    Now, for each $j$ the bromeliad $\cB_j$ is $k$-compatible with $(e_1, \dotsc, e_{j - 1})$ and $e_j \in \cB_j$, so $(e_1, \dotsc, e_\Phi)$ is a $k$-split-degenerate sequence. Recall that $\cH$ does not contain a matching of size $p$, so from the definition of $\Phi$ we have that $(e_1, \dotsc, e_\Phi)$ contains a $(b + 1)$-bromeliad $\cB = (e_{i_1}, \dotsc, e_{i_{b + 1}})$ witnessed by $(\set{C_{i_1}, P_{i_1}}, \dotsc, \set{C_{i_{b+1}}, P_{i_{b+1}}})$.
    Let us note the following set inclusion for future use:
    \begin{equation}\label{eq:C}
        e_{i_2} \cap e_{i_3} = C_{i_3} \subsetneq C_{i_2} = e_{i_1} \cap e_{i_2}. \tag{$\dagger$}
    \end{equation}
    By removing the first edge of $\cB$, we get another $b$-bromeliad $\cB' \coloneqq (e_{i_2}, \dotsc, e_{i_{b + 1}})$ witnessed by $(\set{e_{i_2}, \emptyset}, \set{C_{i_3}, P_{i_3}}, \dotsc, \set{C_{i_{b+1}}, P_{i_{b+1}}})$. Moreover, for $2 \leq j \leq b + 1$, we have $i_j > i_1$ and so $\cB'$ is a $b$-bromeliad in $\cH_{i_1 - 1}$ and is $k$-compatible with $(e_1, \dotsc, e_{i_1 - 1})$. That is, all of the edges in $\cB'$ were available at the time $\cB_{i_1}$ was chosen. To complete the proof, we will show that $\cB' \prec \cB_{i_1}$, thereby contradicting the $\prec$-minimality of $\cB_{i_1}$. 
    
    Firstly, let $C'_{i_1}$ and $P'_{i_1}$ be the petal and core of $e_{i_1}$ in $\cB_{i_1}$ respectively. Since $e_{i_1}$ is not the outer edge of $\cB_{i_1}$, we can be sure that $C'_{i_1}$ is a proper subset of the crown of $\cB_{i_1}$. The crown of $\cB'$ is $e_{i_2} \cap e_{i_3}$. Thus, by \eqref{eq:C}, to prove that $\cB' \prec \cB_{i_1}$ it suffices to show that $e_{i_1} \cap e_{i_2} \subset C'_{i_1}$.

    Now $e_{i_1} = C'_{i_1} \sqcup P'_{i_1}$ and each edge of $\cH_{i_1}$ is disjoint from $P'_{i_1}$. But $e_{i_2} \in \cH_{i_2 - 1} \subset \cH_{i_1}$ which implies that
    \begin{equation*}
        e_{i_2} \cap (e_{i_1} \setminus C'_{i_1}) = e_{i_2} \cap P'_{i_1} = \emptyset, 
    \end{equation*}
    so $e_{i_1} \cap e_{i_2} \subset C'_{i_1}$, which gives the desired contradiction. This establishes that in fact $\chi(\cH, t + \ell) = \chi(\cH, b - 1) \leq \chi_0$, so taking $K = \chi_0$ gives the result.
\end{proof}

\section{The \texorpdfstring{$(t + \ell)$}{t + l}-strong chromatic number of \texorpdfstring{$t$}{t}-intersecting hypergraphs}\label{sec:extremal}

\Cref{thm:general} demonstrates that the only obstacle to finding a $(t + \ell)$-strong colouring of a $t$-intersecting hypergraph with a bounded number of colours is that there may exist a set of $t$ vertices whose link has a high $\ell$-strong chromatic number.
This motivates the definition of a parameter \defn{$\chi(\cH, t, \ell)$} of a hypergraph $\cH$, where
\begin{equation*}
    \chi(\cH, t, \ell) \coloneqq \max \Big(\set{1} \cup \set[\Big]{\chi(\cH_S,\ell) \colon S \in \binom{V(\cH)}{t}}\Big).
\end{equation*}
Since we now know that there is a relationship between $\chi(\cH, t + \ell)$ and $\chi(\cH, t, \ell)$, the natural problem arises of determining the nature of this relationship. As a start, the following statement is a refinement of \cref{thm:mainfiniteness} for $(c - 2)$-intersecting hypergraphs with large $c$-strong chromatic number. 

\begin{theorem}\label{thm:c+2upper}
    For every integer $c\geq 2$, there is a constant $x_c$ such that if $\cH$ is a $(c - 2)$-intersecting hypergraph with $\chi(\cH, c) \geq x_c$, then $\chi(\cH, c) = \chi(\cH, c - 2, 2) + c - 2$.
\end{theorem}

\begin{proof}
    Let $K$ be the value given by applying \cref{thm:general} with $t = c - 2$, $\ell = 2$, $p = 2$, and $\chi = 2c - 1$. We claim that we may set $x_c$ to be $K+1$.
    Since $\cH$ is $(c - 2)$-intersecting and $\chi(\cH, c) \geq x_c > K$, there is some set $S$ of $c - 2$ vertices of $\cH$ such that $\chi(\cH_S,2) \geq 2c$. Let $S$ be the set of $c - 2$ vertices for which $\chi(\cH_S, 2)$ is largest. Note that $\chi(\cH_S) = \chi(\cH_S, 2) = \chi(\cH, c - 2, 2)$.

    First, we show that $\chi(\cH, c)\geq \chi(\cH_S)+c - 2$. Consider the subhypergraph $\cH'$ of $\cH$ consisting of all hyperedges containing $S$. Suppose for contradiction that $\cH'$ admits a $c$-strong colouring with $\chi(\cH_S)+c-3$ colours.
    At most $c - 2$ colours appear on $S$, so there is some colour $z$ that appears on $V(\cH')$ but not on $S$.
    We may assume that the colours used on $S$ are distinct from the colours used on $V(\cH') \setminus S$: if $v \in V(\cH') \setminus S$ receives a colour that appears on $S$, then changing the colour of $v$ to $z$ results in a $c$-strong colouring of $\cH'$ using the same number of colours.
    Thus, the number of colours used by vertices not in $S$ is at most $\chi(\cH_S)-1$.
    However, by the definition of $\chi(\cH_S)$, there is a hyperedge with at most $\abs{S} + 1 = c - 1$ distinct colours, a contradiction. Thus $\chi(\cH, c) \geq \chi(\cH', c) \geq \chi(\cH_S) + c - 2$.

    Now consider a $(\chi(\cH_S) + c - 2)$-colouring $\theta$ which extends an optimal colouring of $\cH_S$ by assigning a unique colour to every vertex in $S$.
    Note that every hyperedge containing $S$ sees at least $c$ colours. 
    Consider a hyperedge $e$ which does not contain $S$. Since $\cH$ is $(c - 2)$-intersecting, this edge $e$ intersects each edge containing $S$ in at least one vertex not in $S$.
    Let $z_1,z_2,\dots, z_{2c}$ be distinct colours assigned by $\theta$ to vertices not in $S$.
    Since the restriction of $\theta$ to $\cH_S$ is an optimal colouring, for each $i\in[c]$ there must be an edge $e_i$ of $\cH_S$ which sees only the colours $z_{2i-1}$ and $z_{2i}$, otherwise these colours could be merged into a single colour.
    But the edge $e$ intersects each such $e_i$, and so $e$ sees at least $c$ colours. 
    Thus $\theta$ is a $c$-strong colouring, and so $\chi(\cH, c) \leq \chi(\cH_S) + c - 2$, as required.
\end{proof}

In light of the preceding result, one might be tempted to conjecture that, for positive integers $t$ and $\ell$, every $t$-intersecting hypergraph $\cH$ of sufficiently large $(t + \ell)$-strong chromatic number satisfies $\chi(\cH, t + \ell)=\chi(\cH, t, \ell)+t$.
The following result demonstrates that this is not the case.
\begin{theorem}\label{thm:construction}
    Given positive integers $t$, $\ell$ and $K$ with $K, \ell\geq 2$, there is a hypergraph $\cH$ with the following properties\textup{:}
    \begin{itemize}
        \item $\cH$ is $t$-intersecting,
        \item $\chi(\cH, t + \ell)= K^{\binom{t + 2\ell - 4}{\ell - 2}} + t + 2\ell - 4$, and
        \item $\chi(\cH, t, \ell)\leq K^{\binom{2\ell - 4}{\ell - 2}} + 2\ell - 4$.
    \end{itemize}
\end{theorem}

\begin{proof}
    To construct such a hypergraph, let $\tau = \binom{t + 2\ell - 4}{\ell - 2}$, and let $A$ and $B$ be disjoint sets of vertices with $\abs{A} = t + 2\ell - 4$ and $\abs{B} = K^\tau$.
    Let $S_1,\dots S_{\tau}$ be the subsets of $A$ of size $t + \ell - 2$, and index the elements of $B$ by writing $B \coloneqq \set{b_{\sigma} \colon \sigma\in [K]^\tau}$.
    Let $\cH$ be the hypergraph whose edge set consists of all sets of the form $S_i\cup \set{b_{\sigma}, b_{\sigma'}}$ such that $\sigma$ and $\sigma'$ differ in their $i$-th coordinate.
    Then $\cH$ is $(t + \ell)$-uniform and every edge contains $t + \ell - 2$ vertices of $A$, so $\cH$ is $t$-intersecting.
    Note that every pair of vertices occur together in some hyperedge of size exactly $t + \ell$, and must therefore receive distinct colours in any $(t + \ell)$-strong colouring. Thus $\chi(\cH, t + \ell) = \abs{V(\cH)} = K^\tau + t + 2\ell - 4$.

    We now turn to colouring the links. Let $S \subseteq V(\cH)$ have size $t$. If $\abs{B \cap S} \geq 3$, then $\cH_S$ has no edges and so $\chi(\cH_S, \ell) = 1$. Next, suppose that $\abs{B \cap S} \in \set{1, 2}$. Rainbow colour $A \setminus S$ and give one further colour to all of $B \setminus S$. Every edge in $\cH_S$ is rainbow-coloured, so
    \begin{equation*}
        \chi(\cH_S, \ell) \leq 1 + \abs{A \setminus S} = 1 + \abs{A} - \abs{S} + \abs{B \cap S} \leq 2 \ell - 1.
    \end{equation*}
    The last possibility is that $\abs{B \cap S} = 0$, meaning $S \subseteq A$. Note that the set $I_S\subseteq [\tau]$ of indices $i$ such that $S\subseteq S_i$ has size exactly $\binom{2\ell - 4}{\ell - 2}$. The equivalence relation $\sim_S$ on $B$, defined by $b_\sigma \sim_S b_{\sigma'}$ if and only if $b_\sigma$ and $b_{\sigma'}$ agree on $I_S$, has $K^{\binom{2\ell - 4}{\ell - 2}}$ equivalence classes. Rainbow colour $A \setminus S$, and for each equivalence class of $\sim_S$ assign one new colour to all vertices in that class. Now every edge in $\cH_S$ is rainbow-coloured, so
    \begin{equation*}
        \chi(\cH_S, \ell) \leq \abs{A \setminus S} + K^{\binom{2\ell - 4}{\ell - 2}} = K^{\binom{2\ell - 4}{\ell - 2}} + 2 \ell - 4.
    \end{equation*}
    Overall, this gives the claimed bound of $\chi(\cH, t, \ell) \leq K^{\binom{2\ell - 4}{\ell - 2}} + 2\ell - 4$.
\end{proof}

We do not know whether the above construction is optimal.
However, we now work toward an upper bound for $\chi(\cH, t + \ell)$ in terms of $\chi(\cH, t, \ell)$ that will show that it is in some sense close to optimal when $t$ is much larger than $\ell$. For this, we need the following immediate corollary of \cref{thm:general}. 

\begin{corollary}\label{cor:excludedsunflower}
    For every pair of integers $p, c\geq 2$, there exists a positive integer $K_{\ref{cor:excludedsunflower}}$ such that every hypergraph $\cH$ either contains a sunflower with $p$ petals and kernel size at most $c - 2$, or $\chi(\cH, c)\leq K_{\ref{cor:excludedsunflower}}$.
\end{corollary}

\begin{proof}
    Apply \cref{thm:general} with $t = c - 1$, $\ell = 1$ and $\chi = 1$.
\end{proof}

\begin{theorem}\label{thm:t+ell-upper}
    For every pair of positive integers $t$ and $\ell$, there is a natural number $K_{\ref{thm:t+ell-upper}}$ such that for any $t$-intersecting hypergraph $\cH$ we have
    \begin{equation*}
        \chi(\cH, t + \ell)\leq \max\set{K_{\ref{thm:t+ell-upper}},(t + \ell) \cdot \chi(\cH, t, \ell)^{\binom{t + \ell - 2}{\ell - 2}}+t + \ell - 2}.
    \end{equation*}
\end{theorem}

\begin{proof}
    Set $K_{\ref{thm:t+ell-upper}}$ to be the constant obtained by taking $c = p = t + \ell$ in \cref{cor:excludedsunflower}.
    By \cref{cor:excludedsunflower}, we may assume that $\cH$ contains a sunflower with $t + \ell$ petals and kernel size at most $t + \ell - 2$. 
    Let $P_1, P_2, \dotsc, P_{t + \ell}$ be the petals of such a sunflower and let $\hat{S}$ be its kernel.
    Consider a $(t + \ell)$-colouring $c_0$ of $\cH$ such that for each petal $P_i$, all vertices in $P_i$ are assigned colour $i$ (the colours of vertices not in any petal can be chosen arbitrarily).
    Note that if a hyperedge intersects every petal, it receives $t + \ell$ colours under this colouring.
    If a hyperedge does not intersect every petal, then, as $\cH$ is $t$-intersecting, it must contain at least $t$ vertices of $\hat{S}$.
    For each $S'\in \binom{\hat{S}}{t}$, let $c_{S'}$ be a $\chi(\cH, t, \ell)$-colouring of $\cH$ whose restriction to $\cH_{S'}$ is $\ell$-strong.
    To obtain a $(t + \ell)$-strong colouring of $\cH$, we colour the vertices not in $\hat{S}$ according to the product colouring of all colourings considered thus far, and then assign $\abs{\hat{S}}$ new colours to the vertices of $\hat{S}$.
    Thus, the total number of colours used is at most $(t + \ell) \cdot \chi(\cH, t, \ell)^{\binom{t + \ell - 2}{t}} + t + \ell - 2$.
\end{proof}

\section{Open problems}

\Cref{thm:mainfiniteness} completes the characterisation of when $\chi(t, c)$ is finite. Nonetheless, the precise value of $\chi(t, c)$ remains open for $t\geq c-1\geq 3$. The best lower bound for $\chi(c - 1, c)$ is $2 c - 1$, which is attained by the complete $(2c - 2)$-uniform hypergraph on $3c - 3$ vertices~\cite{BWY14}. This lower bound is tight for $c \leq 3$, and Blais, Weinstein and Yoshida asked whether it is the correct value for all $c$~\cite[Problem~1.4]{BWY14}. The bound given by our proof is superexponential. It would be interesting to prove a small upper bound.

\begin{problem}\label{prob:UB}
    Improve the upper bound on $\chi(c - 1, c)$. Can it be bounded by an exponential/polynomial/linear function of $c$\textup{?}
\end{problem}

When $t \geq c$, the upper and lower bounds for $\chi(t, c)$ are much closer (since random colourings work) although still not tight. See Blais, Weinstein and Yoshida~\cite{BWY14} and Alon~\cite{Alon2013} for the current state of the art bounds and interesting open problems in this regime. We remark that the current best upper bound for $\chi(c, c)$ is $c^{1/2} e^c$, and so sufficient progress on \cref{prob:UB} would lead to an improvement on this as well.

\Cref{thm:general} gives a qualitative characterisation of when a $t$-intersecting hypergraph has large $(t + \ell)$-strong chromatic number: when there is a set of $t$ vertices whose link has large $\ell$-strong chromatic number. \Cref{thm:c+2upper,thm:t+ell-upper} strengthen this quantitatively by giving bounds on the $(t + \ell)$-strong chromatic number of $\cH$ in terms of the largest $\ell$-strong chromatic number of the link of a set of $t$ vertices. It would be very interesting to refine this bound further.

\begin{problem}
    For fixed positive integers $t$ and $\ell$, what is the smallest integer $x_{t, \ell}$ such that every $t$-intersecting hypergraph $\cH$ satisfies
    \begin{equation*}
        \chi(\cH, t + \ell) = \cO_{t, \ell}\bigl(\chi(\cH, t, \ell)^{x_{t,\ell}}\bigr)?
    \end{equation*}
\end{problem}

\Cref{thm:construction,thm:t+ell-upper} show that $\binom{t + 2\ell - 4}{\ell - 2}/\binom{2\ell - 4}{\ell - 2}\leq x_{t,\ell}\leq \binom{t + \ell - 2}{\ell - 2}$.
Thus, if we fix the value of $\ell$ and allow $t$ to grow we have $x_{t,\ell} = \Theta(t^{\ell - 2})$.
There is a more significant gap if we instead fix the value of $t$, where we see that $x_{t,\ell}$ is $\Omega(1)$ and $\cO(\ell^t)$.

\textbf{Acknowledgements.} This project was initiated at the Graph Theory Workshop held at the Bellairs Research Institute in March 2024. We thank the organisers Sergey Norin, Paul Seymour and David Wood.

{
\fontsize{11pt}{12pt}
\selectfont
	
\hypersetup{linkcolor={red!70!black}}
\setlength{\parskip}{2pt plus 0.3ex minus 0.3ex}

\bibliographystyle{Illingworth.bst}
\bibliography{refs.bib}
}

\end{document}